\documentclass[12pt]{article}
\makeindex
\usepackage{amsfonts,amssymb,amsmath,amscd,amsthm,xcolor,mathtools,tikz-cd,enumerate}
\usepackage[all,cmtip]{xy}
\usepackage[backref=false,colorlinks, linkcolor=blue, citecolor=red]{hyperref}
\usepackage[top=1in, bottom=1.25in, left=1.25in, right=1.25in]{geometry}
\makeatletter

\newtheorem{theorem}{Theorem}[section]
\newtheorem{proposition}[theorem]{Proposition}
\newtheorem{corollary}[theorem]{Corollary}

\newtheorem{lemma}[theorem]{Lemma}
\newtheorem{remark}[theorem]{Remark}
\newtheorem{example}[theorem]{Example} 
\newtheorem{conjecture}[theorem]{Conjecture}

\numberwithin{equation}{section}

\def\Ind{{\rm ind\,}}
\def\Gal{{\rm Gal}}

\title{Totally positive field extensions and the pythagorean index}
\author{P. Mandal, R. Preeti, A. Soman }
\date{}
\begin{document}
	
	\maketitle 
	
	\begin{abstract}
		For a formally real field $F$, we study totally positive field extensions $K$ over $F$. We show that, if $K /F$ is Galois and totally positive then 
		so is the corresponding extension of their pythagorean closures $K_{\rm py}$ over $F_{\rm py}$. We also study the behaviour of weak isotropy and weak hyperbolicity of central simple algebras with an orthogonal involution  over totally positive field extensions. We use some of these results to 
		prove new cases for which a conjecture due to Becher holds. 
	\end{abstract}
	
	\section{Introduction}
	
	Let $F$ be a formally real field. 
	Algebraic structures on such fields are usually studied in relation to the orderings and semiorderings on $F$. 
	A field extension $K/F$ is said to be {\it totally positive} if every semiordering of $F$ extends to a semiordering of $K$ (see \cite[\S3]{Becher}).  
	All the algebraic field extensions of $F$ considered in this paper will be assumed  to be subfields of a fixed 
	algebraic closure $F_{\rm al}$ of $F$. 
	A field $F$ is said to be {\it pythagorean} if a sum of squares in 
	$F$ is a square \cite[Chapter 8, \S$4$]{Lam}.  
	Recall that there exists a smallest pythagorean subfield $F_{\rm py}$ of $F_{\rm al}$ containing $F$, called the {\it pythagorean closure} of $F$ (see \S\ref{prelim}). We start by proving the following property of {\it Galois} totally positive field extensions (see \S\ref{tp extension}).
	\medskip

	\begin{theorem}\label{arbitrary Galois tp} 
		Let $K/F$ be a totally positive Galois extension of formally real fields. Then $K_{\rm py}/F_{\rm py}$ is a totally positive extension.
	\end{theorem}
	\medskip
	
	Though there exist well-known criteria for extending orderings over fields, such results are not known for semiorderings. This makes it difficult to detect if a 
	given field extension $K/F$ is totally positive. In \S\ref{examples}, we provide examples of totally positive field extensions.
	\medskip

	Let $(A,\sigma)$ be a central simple algebra with an orthogonal involution over $F$. In \S\ref{involutions} we study the behavior of $(A,\sigma)$ over totally positive field extensions. In particular, we prove the following.
	
	\begin{theorem}\label{weak hyp involution}
		Let $F$ be a formally real field and let $K$ be a totally positive field extension over $F$. Let $(A,\sigma)$ be a central simple algebra with an orthogonal involution over $F$. If $(A, \sigma)$ is weakly hyperbolic over $K$ then $(A,\sigma)$ is weakly hyperbolic over $F$.
	\end{theorem}
	
	\begin{theorem}\label{anologue-for-involution}
		Let $F$ be a formally real field and let $K$ be a totally positive field extension over $F$. Let $(A, \sigma)$ be a totally decomposable central simple algebra with an orthogonal involution over $F$. If $(A, \sigma)$ is weakly isotropic over $K$ then $(A,\sigma)$ is weakly hyperbolic over $F$.
	\end{theorem}
	\medskip
	
	In the final section, \S\ref{Conjecture-cases}, 
	we use the results from \S\ref{tp extension} of this paper to prove that the following conjecture due to Becher holds for some more cases. This conjecture suggests sufficient conditions on the field extension $K / F$ so  that the pythagorean index remains constant over that extension for the $2$ torsion elements in $Br(F)$. Recall that for a central simple algebra $A$ over $F$, the pythagorean index of $A$, $pind(A)\coloneqq\Ind(A\otimes_FF_{\rm py})$.
	\medskip

	\begin{conjecture}[Becher]\label{conjecture}
		Let $K/F$ be a totally positive field extension. For any central simple algebra $A$ of exponent $2$ over $F$, $pind(A)=pind(A\otimes_FK)$.
	\end{conjecture} 
	
	This conjecture is known to hold (see \cite[Proposition 4.2 and Theorem 4.3]{Becher}) in the following cases: 
	\medskip 
	
	\noindent
	{\bf (i)} $K = F( {\sqrt d} )$, where $d$ is a sum of squares in $F$.\\
	{\bf (ii)} $K=F(X)$, the rational function field in one variable over $F$.\\
	{\bf (iii)}  $K=F((X))$, the field of power series in one variable over $F$.\\
	{\bf (iv)}	 $K$ is the function field of a weakly isotropic quadratic form over $F$.
	\medskip

	In this paper we prove the following.

	\begin{theorem}\label{proof-conjecture}
		Let $F$ be a formally real field and let $K$ be a totally positive field extension over $F$. Let $A$ be a central simple algebra over $F$ of exponent $2$. We prove Becher's conjecture (Conjecture \ref{conjecture}) in the following cases:
		\begin{enumerate}[(i)]
			\item $pind(A)\leq 2$.
			\item $\Ind(A)\leq 4$.
			\item $K/F$ is a Galois extension and $pind(A)\leq 4$.
		\end{enumerate}
	\end{theorem}

\bigskip

\noindent{\it Acknowledgment}: Its a pleasure to thank K. J. Becher and J.-P. Tignol for their comments on an earlier version of this paper.

	\section{Preliminaries and known results}\label{prelim}
	In this section we set up the basic notations to be used throughout this paper. We refer to the relevant sections in \cite{Lam1} and \cite{Lam} for more details. All the fields are assumed to be of characteristic different from $2$ and all the division algebras are finite-dimensional over their centers.
	\medskip

	We denote by $\sum F^2$ the set of elements of a field $F$ that can be expressed as a sum of squares in $F$. 
	A field $F$ is said to be formally real if $-1$ is not a sum of squares in $F$.  
	Hence, $F$ is formally real if $-1 \not\in \sum F^2$. A {\it preordering} of $F$ is a proper subset $T\subsetneq F$ such that $F^2\subseteq T, T+T\subseteq T$ and $T\cdot T\subseteq T$. An ordering $P$ of $F$ is a preordering of $F$ such that $P\cup -P=F$ and $P\cap -P=\{0\}$. By \cite[Chapter 8, Theorem 1.10]{Lam}, $F$ is formally real if and only if $F$ possesses at least one ordering.
	We denote by $X_F$ the set of all orderings on $F$. 
	Given an ordering $P\in X_F$, there exists, a unique up to isomorphism, real closure $F_P$ of $(F,P)$ in $F_{\rm al}$ (see \cite[Chapter 8, Theorem 2.8]{Lam}). 
	
	\medskip 
	
	A {\it semiordering} of a field $F$ is a subset $S\subset F$ such that $1\in S$, $F^2S\subset S$, $S+S\subset S$, $S\cup -S=F$ and $S\cap -S=\{0\}$ (see \cite[page 5]{Prestel}).
	Thus, every ordering of a field is a semiordering. But every semiordering need not be an ordering. For instance, let  $F=\mathbb{Q}((X_1))((X_2))\cdots((X_n))$ be an iterated Laurent series field over the field of rational numbers with $n\geq 3$. Consider the preordering $T= \sum F^2$ of $F$. Let $S=T\cup\{-a:a\in P\setminus T\}$, where $P$ is an ordering of $F$ containing $T$. By \cite[Proposition 14.12]{Lam1}, $S$ is a semiordering but not an ordering.
	
	\medskip 
	
	The field $F$ is said to be {\it pythagorean} if $F^2+F^2\subset F^2$, i.e., $\sum F^2=F^2$. There exists a smallest pythagorean subfield of $F_{\rm al}$ that contains $F$ (see \cite[Chapter 8, \S 4]{Lam}). This field is called the {\it pythagorean closure} of $F$ and is denoted by $F_{\rm py}$. 
	The pythagorean closure of $F$ can also be explicitly constructed as follows: let $\mathcal{F}$ be the family of extensions $K/F$, 
	$K \subset F_{\rm al}$, for which there exists a tower  
	\[
	F = K_0\subset K_1\subset\cdots\subset K_n =K,
	\]
	such that $K_{i+1}=K_i(\sqrt{1+a_i^2}), \text{ where } a_i\in K_i$. The compositum of all the fields in $\mathcal{F}$ is $F_{\rm py}$.
	\medskip 
	
	Let $q$ be a quadratic form over a field $F$ and let $V$ be its underlying vector space. The dimension $\dim q$ of $q$, is the dimension of the vector space $V$ over $F$. Let $b_q$ be the associated bilinear form to $q$.
	If $b_q$ is non-degenerate then $q$ is said to be {\it non-degenerate} or {\it regular}. In this paper, by a 
	quadratic form we mean a non-degenerate quadratic form. 
	A quadratic form $q$ of dimension $n$ is equivalent to a diagonal form $a_1X_1^2+\cdots +a_nX_n^2$ over $F$. The diagonal form $\sum_{i=1}^{n} a_iX_i^2$ is denoted by $\langle a_1,\ldots,a_n\rangle$.  
	We denote by $W(F)$ the Witt ring of $F$ (see \cite[Chapter 2]{Lam} for details).
	\medskip 
	
	A quadratic form $q$ is said to be {\it isotropic} over $F$, if there exists a non-zero vector $v \in V$ such that $q(v) = 0$. 
	A quadratic form $q$ over $F$ is said to be {\it weakly isotropic} if there exists a natural number $n\in\mathbb{N}$, such that the $n$ fold orthogonal 
	sum of $q$, $n \cdot q$ is isotropic. 
	For quadratic forms over arbitrary fields, isotropy is {\it not} the same as weak isotropy. 
	For a pythagorean field $F$ both these notions match i.e., every weakly isotropic quadratic form $q$ over $F$ is isotropic over $F$. 
	\medskip
	
	We next recall the notion of signature of a quadratic form with respect to orderings and semiorderings. 
	Let $q=\langle a_1,\ldots,a_n\rangle$ be a quadratic form over $F$ and $P\in X_F$ be an ordering on $F$. Let $F_P$ be the real closure of $F$ at $P$. Let $r$ be the number of $a_i$'s that are positive and $s$ be the number of $a_i$'s that are negative with respect to $P$. Then the {\it signature} of $q$ at $P$, denoted by ${\rm sign}_P (q)$, is equal to $r-s$. A quadratic form $q$ is called {\it definite} at $P$ if $|{\rm sign}_P(q)|=\dim q$ and {\it indefinite} at $P$ otherwise.  A quadratic form is called {\it totally indefinite} if it is indefinite at every ordering $P\in X_F$. Thus, if $q$ is weakly isotropic then it is totally indefinite. The converse is true if $F$ satisfies the {\it strong approximation property} (SAP) (see \cite[\S 9]{Prestel} for more details). 
	
	The {\it total signature} map ${\rm sign}$ is defined by 
	
	\[{\rm sign}: W(F)\to \prod_{P\in X_F}W(F_P) \]
	
	\[ q\mapsto \big({\rm sign}_P(q)\big)_P . \] 
	
	The kernel of this map, $\ker({\rm sign})$ is denoted by $W_t(F)$.  
	Every element of $W_t(F)$ is $2$-primary torsion and $W_t(F)=\ker\big(W(F)\to W(F_{\rm py}) \big)$ (see \cite[Chapter 8, Theorem 3.2 and Theorem 4.10]{Lam}). In particular, $W_t(F_{\rm py})=0$ (see \cite[Chapter 8, Theorem 4.1(1)]{Lam}). The signature map is defined more generally for $(A,\sigma)$, a central simple algebra with an orthogonal involution over $F$ (see \cite[\S 11]{boi}). Let $Trd_A$ denote the reduced trace of $A$. For $P\in X_F$ the signature of $\sigma$ at $P$ is defined by \[
	{\rm sign}_P\ \sigma:= \sqrt{{\rm sign}_P\ T_\sigma},\] where $T_\sigma$ is the \emph{involution\ trace\ form} of $(A,\sigma)$, which is a quadratic form on $A$ defined by \[
	T_\sigma(x):=Trd_A(\sigma(x)x), \text{for all }x \in A.\]
	\medskip 
	
	Let $S\subset F$ be a semiordering.  
	The signature of a quadratic form $q=\langle a_1,\ldots,a_n\rangle$ at $S$ is defined as 
	\[
	{\rm sign}_S(q)= (\text{number of }a_i\in S) -  (\text{number of } a_i\in -S).
	\] 
	A quadratic form $q$ is said to be {\it definite} with respect to the semiordering $S$ if $|{\rm sign}_S(q)|=\dim q$ and it is said to be {\it indefinite} with respect to $S$ if $|{\rm sign}_S(q)|<\dim q$. The connection between 
	semiorderings and weak isotropy is given by the following result due to Prestel (see \cite[Theorem 2.9]{Prestel}). 
	
	\begin{theorem} (Prestel) 
		A quadratic form $q$ is weakly isotropic over $F$ if and only if it is indefinite with respect to every {\it semiordering} of $F$. 
	\end{theorem} 
	
	We next collect some basic facts on central simple algebras over a field. 
	Let $D$ be a central division algebra over $F$. The positive square root of $\dim_F D$ is called the {\it degree} of $D$. It is denoted by $\deg D$. Let $A$ be a central simple algebra over a field $F$. By the Artin-Wedderburn theorem, $A\simeq M_n( D)$, for some natural number $n$ and 
	division algebra $D$ over $F$. The {\it index} of $A$ is defined as $\Ind(A)\coloneqq\deg D$. In particular, $\Ind D=\deg D$. We denote by $Br(F)$ the Brauer group of $F$. If central simple algebras $A$ and $B$ over $F$ represents the same equivalence class in $Br(F)$, then we write $A\sim B\in Br(F)$. If $K/F$ is an extension of fields then the central simple algebra $A\otimes_FK$ over $K$ is denoted by $A_K$. 
	\medskip 
	
	The order of $A$ in $Br(F)$ is called the {\it exponent} of $A$ and it is denoted by $\exp(A)$. We denote by ${}_nBr(F)$ 
	the $n$-torsion elements in $Br(F)$, i.e., the subgroup of $Br(F)$ consisting of elements of exponent dividing $n$. In particular, ${}_2Br(F)$ is the subgroup of elements of exponent at most $2$. We refer to \cite[\S 9]{D} for more details. 
	\medskip

	A quaternion $F$-algebra $\big(\frac{a,b}{F}\big)$ is an $F$-algebra of degree $4$, 
	generated by $i,j$ with the relations: $i^2=a,j^2=b$ and $ij=-ji$, where $a,b\in F\setminus\{0\}$. The quadratic form $\langle 1,-a,-b,ab\rangle$ is called the \emph{norm} form of $\big(\frac{a,b}{F}\big)$. 
	It is well known that $\big(\frac{a,b}{F}\big)$ is a division algebra if and only if the associated norm form is anisotropic (see \cite[Chapter 3, Theorem 2.7]{Lam}).
	\medskip

	
	

	\bigskip 
	
	\medskip

	\section{Totally positive field extensions}\label{tp extension}
	In this section, we study totally positive field extensions. We show that, if a field extension $K/F$  is {\it Galois} and totally positive then $K_{\rm py}/F_{\rm py}$ is totally positive. 
	We recall the following results proved in \cite{Becher}, \cite{BLS} which will be used in this paper. 
	\medskip

	\begin{proposition} (\cite[Proposition 3.7]{BLS}) \label{equivalent conditions for quadratic} 
		Let $K=F(\sqrt{d})$ of $F$ be a quadratic extension. Then the following are equivalent: 
		\begin{enumerate}
			\item $K/F$ is totally positive.
			\item $d\in\sum F^2\setminus F^2$.
			\item Every ordering of $F$ extends to an ordering of $K$.
		\end{enumerate}
	\end{proposition}

	\begin{remark}\label{totally positive over totally positive} It is easy to see that 
		if $K/F$ is totally positive and $E/K$ is totally positive, then $E/F$ is totally positive.
	\end{remark}

	\medskip 
	
	The next result from \cite{Becher} characterizes the notion of total positiveness in terms of isotropy of quadratic forms. 
	\medskip 
	
	\begin{lemma} (\cite[Lemma 3.1]{Becher})\label{totally positive criterion}
		For a field extension $K/F$, the following are equivalent:
		\begin{enumerate}
			\item $K/F$ is totally positive.
			\item For any quadratic form $\varphi$ over $F$, if $\varphi_K$ is isotropic, then $\varphi$ is weakly isotropic over $F$.
		\end{enumerate}
	\end{lemma} 
	
	The following lemma is proved in \cite{Becher}. 
	\medskip 
	
	\begin{lemma} (\cite[Corollary 3.3]{Becher}) \label{py totally positive}
		The pythagorean closure $F_{\rm py}$ of $F$ is a totally positive extension of $F$.
	\end{lemma}	
	%
	%
	
	\medskip

	Let $K/F$ be a finite separable extension. Consider the trace bilinear form of the extension, ${\rm tr}:K\times K\to F$, given by ${\rm tr}(x,y)=tr_{K/F}(xy)$, where $tr_{K/F}$ is the trace of the field extension $K/F$. 
	Recall that a field $F$ is said to have the {\it strong approximation property} (SAP), if every semiordering on $F$ is an ordering. 
	We have the following lemma. 
	\medskip

	\begin{lemma}\label{semiordering is ordering}
		Let $F$ be a formally real field such that $F$ has SAP. Let $K/F$ be a finite field extension. Then the following are equivalent.
		\begin{enumerate}[(i)]
			\item The field extension $K/F$ is totally positive. 
			\item Every ordering of $F$ extends to an ordering of $K$. 
			\item For every $P\in X_F$, ${\rm sign}_P({\rm tr})>0$.
		\end{enumerate}
	\end{lemma}
	\begin{proof}
		The implication $(i) \implies (ii)$ follows from the definition and $(ii) \implies (iii)$ follows from \cite[Chapter 3, Lemma 2.7]{Scha}. 
		Next, suppose $(iii)$ holds. By \cite[Chapter 3, Lemma 2.7]{Scha}, every ordering of $P$ extends to an ordering of $K$. As $F$ has SAP, every semiordering on $F$ is an ordering and hence $(i)$ is proved.
	\end{proof}
	\medskip

	\medskip 
	
	\begin{lemma} \label{2extension}
		Let $K/F$ be an extension of formally real fields. If $K/F$ is a totally positive field extension then for every subfield $F\subseteq E\subseteq K$, $E/F$ is also totally positive. 
	\end{lemma}

	\begin{proof} Let $q$ be a quadratic form over $F$ which becomes isotropic over $E$. In particular, $q$ becomes isotropic over $K$. As $K/F$ is totally positive, $q$ is weakly isotropic over $F$. Hence $E/F$ is totally positive by (Lemma \ref{totally positive criterion}).
		\medskip

	\end{proof}
	
	\medskip 
	
	\begin{proposition}\label{finite}
		Let $K/F$ be a totally positive field extension. Then, $K_{\rm py}/F_{\rm py}$ is totally positive if and only if following condition is satisfied: \begin{gather*}
			(*)\quad LE/E \text{ is totally positive for any finite subextensions }\\ L/K \text{ and } E/F \text{ with } L\subset K_{\rm py} \text{ and } E\subset F_{\rm py}.
		\end{gather*}
	\end{proposition}
	\begin{proof} We have
		\[\begin{tikzcd}
			K\arrow[no head, to=1-2]&L\arrow[no head,to=1-3]&LE\arrow[no head,to=1-4]&K_{\rm py}\\
			F\arrow[no head,to=2-3]\arrow[no head,to=1-1]&&E\arrow[no head,to=2-4]\arrow[no head,to=1-3]&F_{\rm py}\arrow[no head,to=1-4]
		\end{tikzcd}\] 
		Suppose that $K_{\rm py}/F_{\rm py}$ is totally positive. Let $q$ be a quadratic form over $E$ which is isotropic over $LE$. Then $q$ is isotropic over $K_{\rm py}$. Since $K_{\rm py}/F_{\rm py}$ is assumed to be totally positive, $q$ is weakly isotropic over $F_{\rm py}$. Furthermore, $F_{\rm py}/E$ is totally positive since $E_{\rm py}=F_{\rm py}$. Hence, $q$ is weakly isotropic over $E$. 
		Thus, $LE/E$ is totally positive by (Lemma \ref{totally positive criterion}).
		\medskip 
		
		Conversely, assume that the condition $(*)$ is satisfied. Let $q$ be a quadratic form over $F_{\rm py}$ which becomes isotropic over $K_{\rm py}$. Choose a finite subextension $E/F$ of $F_{\rm py}/F$ such that $q$ is defined over $E$. Let $L/K$ be a finite field extension containing the field $KE$ and the coefficients of an isotropic vector of $q$ over $K_{\rm py}$. Then $q$ is isotropic over $L$. By the condition $(*)$, $q$ is weakly isotropic over $E$ and hence $q$ is weakly isotropic over $F_{\rm py}$. Thus, $K_{\rm py}$ over $F_{\rm py}$ is totally positive. 
	\end{proof}
	
	\medskip

	\begin{proposition}\label{finite Galois tp}
		Let $K/F$ be a totally positive finite Galois extension of formally real fields. Then $KF_{\rm py}/F_{\rm py}$ is a totally positive extension. 
	\end{proposition}
	\begin{proof}		
		Let $K=F(\theta)$ and let $f(t)\in F[t]$ be the minimal polynomial of $\theta$ over $F$. 		
		By \cite[Proposition 3.5]{BLS} and \cite[Chapter 3, Lemma 2.7]{Scha}, $f(t)$ has $[K:F]$ number of real roots, i.e., all the roots of $f(t)$ are in 
		a real closure $F_P$ of $F$. 
		Let $g(t)$ be the minimal polynomial of $\theta$ over $F_{\rm py}$. As 
		$g(t)$ divides $f(t)$, all the roots of $g(t)$ also lie in $F_P$. Since $F_{\rm py}/F$ is totally positive (see Lemma \ref{py totally positive}), there exists an ordering $\widetilde{P}$ of $F_{\rm py}$ extending $P$. As $F_{\rm py}/F$ is algebraic, we have $\left(F_{\rm py}\right)_{\widetilde{P}}=F_P$. Hence all the roots of $g(t)$ lie in $\left(F_{\rm py}\right)_{\widetilde{P}}$.
		Therefore, $KF_{\rm py}/F_{\rm py}$ being Galois, it is totally positive by \cite[Chapter 3, Lemma 2.7]{Scha} and \cite[Theorem 3.9]{BLS}.
	\end{proof}
	\medskip

	
	We now prove (Theorem \ref{arbitrary Galois tp}).
	\medskip
	
	\noindent{\bf Proof of  Theorem \ref{arbitrary Galois tp}}. We first show that $KF_{\rm py}/F_{\rm py}$ is totally positive. 
	As $K / F$ is Galois, we have $KF_{\rm py}/F_{\rm py}$ is Galois and we identify the Galois group $\Gal(KF_{\rm py}/F_{\rm py})$ with a subgroup of $\Gal(K/ F)$. In fact, there is a natural isomorphism of $\Gal (K / K \cap F_{\rm py})$ with $ \Gal (KF_{\rm py} / F_{\rm py} )$. 
	
	\[\begin{tikzcd}
		K\arrow[no head, to=1-2]\arrow[no head, from=1-1, to=2-1]&KF_{\rm py}\arrow[no head, to=1-3]\arrow[no head, from=1-2, to=2-2]&K_{\rm py}\arrow[no head, to=2-3]\\
		F\arrow[no head, to=1-1, to=2-2]&F_{\rm py}\ar[equal]{r}&F_{\rm py}\arrow[no head, to=1-3]
	\end{tikzcd}\]
	
	To show that $KF_{\rm py}/F_{\rm py}$ is totally positive, by \cite[Corollary 3.2]{BLS},
	it is enough to show that any finitely generated subextension $L / F_{\rm py}$ of $KF_{\rm py} / F_{\rm py} $ is totally positive. 
	As $KF_{\rm py} / F_{\rm py}$ is Galois, $L / F_{\rm py}$ is {algebraic} and finitely generated. Hence $L / F_{\rm py}$ is a {\it finite} field extension.
	\medskip
	
	\noindent{\bf case 1.} For $L$ as in the above paragraph, suppose further $L/F_{\rm py}$ is Galois.
	We show that $L / F_{\rm py}$ is totally positive. 
	Let $H = \Gal(KF_{\rm py} / L)$.  
	By the above identification of Galois groups, we suppose $H$ is a subgroup of $\Gal (K / F)$. 
	Let $M=K^{H} \subseteq K$. 
	As $M/(K\cap F_{\rm py})$ is a finite separable field extension, we choose $\theta \in M$ such that $M=(K\cap F_{\rm py})(\theta)$. 
	Further, as $\theta$ is algebraic over $F$, $F(\theta) /F$ is a finite field extension. 
	Let $\widetilde{M}\subseteq K$ be the Galois closure of $F(\theta)$ over $F$. So $\widetilde{M}/F$ is also a finite field extension. We have the following: 
	
	\[\begin{tikzcd}
		\widetilde{M}\arrow[no head,to=2-1]\\
		F(\theta)\arrow[no head, to=2-2]\arrow[no head, from=2-1, to=3-1]& M\coloneqq(K\cap F_{\rm py})(\theta)\arrow[no head, to=2-3]\arrow[no head, from=2-2, to=3-2]&MF_{\rm py}=L\\
		F\arrow[no head, to=3-2]& K\cap F_{\rm py}\arrow[no head, to=3-3]& F_{\rm py}\arrow[no head, to=2-3]
	\end{tikzcd}\]
	Since $K/F$ is totally positive, the finite Galois subextension $\widetilde{M}/F$ is also totally positive (see Lemma \ref{2extension}). By  (Proposition \ref{finite Galois tp}) above, $\widetilde{M}F_{\rm py}/F_{\rm py}$ is totally positive. As 
	$MF_{\rm py}/F_{\rm py}$ is a subextension of $\widetilde{M}F_{\rm py}/F_{\rm py}$, it is totally positive by (Lemma \ref{2extension}). Since $M=(K\cap F_{\rm py})(\theta)$ and $L=MF_{\rm py}$, we get $L/F_{\rm py}$ is totally positive. 
	Thus, if $L / F_{\rm py}$ is a finite Galois field extension then it is totally positive. 
	\medskip 
	
	\noindent{\bf case 2.} Let $E/F_{\rm py}$ be any finite degree subextension of $KF_{\rm py}/F_{\rm py}$. We show that $E/F_{\rm py}$ is totally positive. Let $L$ be the Galois closure of $E/F_{\rm py}$. Then, $L\subset KF_{\rm py}$ as  $KF_{\rm py}/F_{\rm py}$ is a 
	Galois field extension. Further, $L/F_{\rm py}$ is a Galois field extension of finite degree. By the above case 1., we see that $L/F_{\rm py}$ is totally positive. Hence, by (Lemma \ref{2extension}), $E/F_{\rm py}$ is also totally positive. 
	Thus $KF_{\rm py}/F_{\rm py}$ is totally positive.
	\medskip 
	
	We next show that  $K_{\rm py}/F_{\rm py}$ is totally positive, by using  $KF_{\rm py} / F_{\rm py}$ is totally positive. 
	We start by showing that ${(KF_{\rm py})}_{\rm py} = K_{\rm py}$. 
	As $K\subset KF_{\rm py}$, hence $K_{\rm py}\subset \big(KF_{\rm py}\big)_{\rm py}$. On the other hand, $KF_{\rm py}\subset K_{\rm py}$, hence $\big(KF_{\rm py}\big)_{\rm py}\subset K_{\rm py}$. Therefore, by 
	(Lemma \ref{py totally positive}), $K_{\rm py}/KF_{\rm py}$ is totally positive. 
	As we have already shown that $KF_{\rm py} / F_{\rm py}$ is totally positive, 
	by (Remark \ref{totally positive over totally positive}), we have, 
	$K_{\rm py}/F_{\rm py}$ is totally positive.\qed
	
	\medskip 
	
	\begin{remark}
		Let $K/F$ be a totally positive Galois extension of formally real fields. Then by the above (Theorem \ref{arbitrary Galois tp}), it follows that every anisotropic quadratic form over $F_{\rm py}$ remains anisotropic over $K_{\rm py}$.
	\end{remark}
	
	In the above (Theorem \ref{arbitrary Galois tp}), the hypothesis on the field extension $K / F$ being {\it Galois} totally positive is necessary, as can be seen from the following proposition.

	\begin{proposition} \label{root2 example} 
		Let $F = \mathbb{Q}$ and 
		let $K = \mathbb{Q}(\sqrt[4]{2})$ be the field extension given by a real fourth root of 
		$2$. Then $K /F$ is an algebraic totally positive field extension such that $K_{\rm py} / F_{\rm py}$ is not totally positive. 
	\end{proposition}
	\begin{proof}		
		We first show that $K/F$ is totally positive. As $F$ has a unique ordering, by \cite[Corollary 14.8]{Lam1}, 
		$F$ has SAP, i.e., every semiordering is an ordering. As $K$ can be embedded in $\mathbb{R}$, the unique ordering of $F$ extends to an ordering of $K$. Hence by (Lemma \ref{semiordering is ordering}) $K / F$ is totally positive.
		\medskip
		
		Let $E=\mathbb{Q}(\sqrt{2})$. We now show that the quadratic extension $K/E$ is not totally positive. By (Proposition \ref{equivalent conditions for quadratic}), $K/E$ is totally positive if and only if $\sqrt{2}$ is a sum of squares in $E$, i.e., $\sqrt{2}$ is positive with respect to every ordering of $E$. On the other hand, by \cite[Chapter 8, Example 1.13]{Lam}, there is an ordering of $E$ with respect to which $\sqrt{2}$ is negative. 
		Hence $\sqrt{2}$ is not a sum of squares in $E$ and $K/E$ is not totally positive.
		\medskip 
		
		Clearly, $F\subset E\subset F_{\rm py}$ and $E_{\rm py}=F_{\rm py}$. We show that $K_{\rm py}/F_{\rm py}$ is not totally positive by using (Proposition \ref{finite}) with $L=K$. So, $LE/E$ is the same as $K/E$, which is not totally positive. This proves the proposition.
	\end{proof}
	%

	\medskip

	\begin{proposition}\label{quadratic totally positive}
		Let $K/F$ be a Galois totally positive field extension. Then $K(\sqrt{d})$ over $ F(\sqrt{d})$ is totally positive, for any $d \in \sum F^2\setminus F^2$.
	\end{proposition}
	\begin{proof} 
		As $K /F$ is Galois and totally positive, by (Theorem \ref{arbitrary Galois tp}), we have $K_{\rm py}/F_{\rm py}$ is totally positive. Further, by the hypothesis on $d$, we have 
		$F(\sqrt{d}) \subset F_{\rm py}$ and $K(\sqrt{d}) \subset K_{\rm py}$. 
		
		\[\xymatrix{ 
			K \ar@{-}[r] & K({\sqrt{d}}) \ar@{-}[r] & K_{\rm py}  \\ 
			F \ar@{-}[r] \ar@{-}[u] & F({\sqrt{d}}) \ar@{-}[r] \ar@{-}[u] & F_{\rm py} \ar@{-}[u] \\}
		\]
		Hence, by the above (Proposition \ref{finite}), we have $K(\sqrt{d})/F(\sqrt{d})$ is totally positive.
	\end{proof} 
	
	\medskip

	\begin{remark} The hypothesis $K/ F$ being {\it Galois} totally positive is necessary in the above (Proposition \ref{quadratic totally positive}). We consider the algebraic field extension 
		$K = \mathbb{Q}(\sqrt[4]{2})$ over $F = \mathbb{Q}$. By (Proposition \ref{root2 example}), we see that $K/F$ is totally positive. But as seen in the proof of (Proposition \ref{root2 example}), for $d=2$, $K(\sqrt{d})$ over $ F(\sqrt{d})$ is not totally positive. 
	\end{remark} 
	\begin{remark} As shown above, for $K = \mathbb{Q}(\sqrt[4]{2})$ and $F= \mathbb{Q}$, we have $K/F$ is totally positive (see Proposition \ref{root2 example}). Let $F(X)$ denote the rational function field in one variable over $F$. Then $K(X) / F(X)$ is not totally positive (see \cite[Corollary 3.14]{BLS}). Hence the notion of total positiveness is not ``stable'' in the sense that it does not remain invariant over purely transcendental extensions of degree $1$. 
	\end{remark}
	
	\section{Examples of totally positive extensions}\label{examples}
	
	We next discuss some examples of totally positive field extensions. 
	We start with the following well known examples. 
	\medskip 
	
	\begin{example}\label{ex1}\normalfont
		\noindent	\begin{enumerate}
			
			\item By Springer's theorem \cite[Chapter 7, Theorem 2.7]{Lam}, every odd degree field extension is totally positive.
			\item The rational function field $F(x)$ over $F$ is totally positive (see \cite[Chapter 9, Lemma 1.1]{Lam}). 
			This implies that if $K = F(x_1, \ldots, x_n)$ is a purely transcendental field 
			extension of $F$ then $K/F$ is totally positive. 
			\item By Springer's theorem for complete discrete valued fields (see \cite[Chapter 6, Theorem 1.4]{Lam}), the field $K=F((x))$ of Laurent series over $F$ is totally positive. 
			\item Let $q$ be a weakly isotropic form over $F$ and $K=F(q)$ be a rational function field of $q$. Then $K/F$ is totally positive by \cite[Theorem 4.3]{Becher}. 
		\end{enumerate}
		
	\end{example}
	
	\medskip

	\medskip 
	
	\begin{proposition}
		Let $F$ be a formally real algebraic field extension of $\mathbb{Q}$. A finite field extension $K/F$ is totally positive if and only if every ordering of $F$ extends to an ordering of $K$.  In particular, a formally real number field is totally positive over $\mathbb{Q}$.
	\end{proposition}
	\begin{proof} We note that $F$ has SAP (see \cite[Corollary 14.8]{Lam1}). The first part of the proposition now follows from the above (Lemma \ref{semiordering is ordering}). 
		\medskip 
		
		Suppose next that $K$ a formally real number field. To show that $K$ over $\mathbb Q$ is a totally positive field extension, we start 
		with an ordering $P$ of $K$. Clearly, $P\cap\mathbb{Q}$ is an ordering of $\mathbb{Q}$. 
		If $\sigma$ is an embedding of $K$ into $\mathbb{R}$, then by pulling back the ordering $\mathbb{R}^2$ of $\mathbb{R}$, we get an ordering of $K$. As $\mathbb{Q}$ is a uniquely ordered field with the ordering $\mathbb{R}^2\cap\mathbb{Q}$, we have $\mathbb{R}^2\cap\mathbb{Q}=P\cap\mathbb{Q}$. In particular, every ordering of $\mathbb{Q}$ has an extension to an ordering of $K$. As $\mathbb{Q}$ has SAP, using (Lemma \ref{semiordering is ordering}) we conclude the proof.
	\end{proof}
	\medskip 
	

	\begin{proposition}[Generic splitting field of a quadratic form]
		Let $q$ be a quadratic form of even dimension over $F$. The generic splitting field $K$ of $q$ over $F$ is totally positive if and only if $q$ is torsion.
	\end{proposition}
	\begin{proof}
		Let $F=F_0\subset F_1\subset F_2\subset\cdots\subset F_h=K$ be the generic splitting tower of $q$, and let $q_0=q_{\rm an}$ and $F_i=F_{i-1}(q_{i-1})$ where $q_i=(q_{i-1})_{\rm an}$ is the anisotropic part of $q_{i-1}$ over $F_i$. Then $\dim q_h\leq 1$. 
		Since $q$ is an even dimensional form, the $K$-anisotropic part of $q$ is zero. 
		\medskip

		If $K/F$ is totally positive then for every ordering $P$ of $F$ there is an ordering ${\widetilde P}$ of $K$ extending $P$. Hence the real closure $F_P \subset K_{{\widetilde P}}$ and ${\rm sign}_P(q)= {\rm sign}_{{\widetilde P}} (q_h) = 0$. Hence $q$ is a torsion form (see \cite[Chapter 8, Theorem 3.2]{Lam}).
		\medskip 
		
		Conversely, let $q$ be a torsion form. So each $q_i$ ($1\leq i\leq h$) is torsion. Hence, by (Example \ref{ex1}($4$)), for every $i$, $ 1 \leq i \leq h$ the field extension $F_i/F_{i-1}$ is totally positive. In particular, $K/F$ is totally positive by (Remark \ref{totally positive over totally positive}).	
	\end{proof}
	
	
	Let $A$ be a central simple algebra over $F$ and $F(A)$ be the function field of the Severi-Brauer variety of $A$. Recall that $A$ is said to be locally split over $F$ if $A_{F_P}=0$ for all orderings $P \in X_F$.
	\begin{theorem}
		Let $F$ be a formally real field and let $A$ be a central simple algebra over $F$ of any degree. Then $A$ is locally split over $F$ if and only if every ordering of $F$ can be extended to $F(A)$. 
	\end{theorem}
	\begin{proof}
		Suppose $A$ is locally split over $F$. Let $P \in X_F$. As $A_{F_P}=0$, $F(A_{F_P})$ is a purely transcendental field extension over $F_P$ (see \cite[Theorem 13.11]{Sal}). As $F(A)\cdot F_P\subset F(A_{F_P})$ and $F(A_{F_P})$ is formally real, we have $F(A)$ is also formally real. Moreover, the unique ordering of $F_P$ can be extended to an ordering $S$ of $F(A_{F_P})$ (see \cite[Chapter 8, Example 1.13(C)]{Lam}). Then $S \cap F(A)$ is an ordering of $F(A)$ extending the ordering $P$ of $F$. Thus every ordering of $F$ can be extended to $F(A)$.
		\vspace{.2cm}
		
		Suppose now that every ordering of $F$ can be extended to $F(A)$. This implies that $F(A)$ is formally real. Let $P \in X_F$. We will show that $A_{F_P} =0$. Suppose $A_{F_P} \neq 0$. Then $A_{F_P}= \big( \frac{-1,-1}{F_P}\big)$. As $F(A)$ is the function field of the Severi-Brauer variety of $A$, we have $\big( \frac{-1,-1}{F(A)\cdot F_P}\big)=0$ by \cite[Corollary 13.9 and Theorem 13.11]{Sal}.
		Hence, $\langle1,1,1,1\rangle=0$ over $F(A)\cdot F_P$ and so $F(A)\cdot F_P$ is not a formally real field.
		
		On the other hand, if $S \subseteq F(A)$ is an ordering extending $P \subseteq F$, then $F_P \subseteq (F(A))_S$. Hence, $F(A)\cdot F_P \subseteq (F(A))_S$, which is formally real. This is a contradiction as subfields of formally real fields are also formally real. Hence $A_{F_P }=0$ for all $P \in X_F$, i.e., $A$ is locally split over $F$.
	\end{proof}
	
	\medskip 
	
	The field invariant $\hat{u}(F)$ of a field $F$ is defined as follows 
	\[\hat{u}(F)=\sup\{\dim q:q\text{ is anisotropic and weakly isotropic} \}.\]
	We show below that the $\hat{u}$-invariant does not behave well under totally positive field extensions. 
	
	\begin{proposition} 
		There exists a totally positive odd degree field extension $K/F$ such that $\hat{u}(K)=\infty$ and $\hat{u}(F)=0$.
	\end{proposition}

	\begin{proof} 
		Let $E$ be a formally real pythagorean field. Hence, every weakly isotropic form over $E$ is isotropic. In particular, $\hat{u}(E)=0$. Let $L/E$ be an odd degree field extension such that
		$L$ is not pythagorean. Let $F=E((X))$ and $K=L((X))$ be Laurent series fields in one variable over $E$ and $L$, respectively. The field $F$ is pythagorean 
		(see \cite[Chapter 8, Proposition 4.11]{Lam}) and hence, $\hat{u}(F)=0$. By Springer's theorem (see Example \ref{ex1}($1$)), $K/F$ is totally positive. Let $d\in \sum L^2 \setminus L^2$. For every $n\in\mathbb{N}$, consider the quadratic form $q_n=n\cdot\langle 1\rangle\perp \langle X,-dX\rangle$ over $K$. Then each $q_n$ is anisotropic and is weakly isotropic over $K$. Therefore, $\hat{u}(K)=\infty$.
	\end{proof}

	\section{Weak isotropy and Weak hyperbolicity over totally positive extensions}\label{involutions}
	Let $K/F$ be a totally positive field extension and $q$ be a quadratic form over $F$. If $q$ is isotropic over $K$ then we know that $q$ is weakly isotropic over $F$. In this section, we look at the analogous question for algebras with involution.
	
	Let $(A,\sigma)$ be a central simple $F$-algebra with an orthogonal involution. Then $(A, \sigma)$ is \emph{isotropic} over $F$ if there exists a nonzero $x \in A$ such that $ \sigma(x)x=0$ and it is said to be \emph{weakly isotropic} if there exist nonzero $x_1,\ldots, x_n \in A$ such that
	$\sum_{i=1}^{n}\sigma(x_i)x_i=0$ (see \cite{ul}). If there exists an element $e \in A$ with $e^2=e$ and $\sigma(e)=1-e$ then $(A,\sigma)$ is said to be $hyperbolic$ (see \cite[Definition 6.8]{boi}). Similarly, 
	($A, \sigma)$ is said to be $weakly\ hyperbolic$ if there exists an $n \in \mathbb{N}$ such that $(M_n(F),t) \otimes_F (A, \sigma)$
	is hyperbolic, where $t$ denotes the transpose involution, see (\cite{ul}).
	
	
	\begin{lemma} \label{weak hyperbolic signature}
		Let $(A,\sigma)$ be a central simple algebra over $F$ with an orthogonal involution. If $\sigma$ is weakly hyperbolic over $F_P$, for all $P \in X_F$ then ${\rm sign}_P\ \sigma=0$ for all $P \in X_F$.
	\end{lemma}
	\begin{proof}
		Let $P\in X_F$. If $A_{F_P}$ is split then $\sigma$ corresponds to a quadratic form $q_P$ which is weakly hyperbolic. Hence, ${\rm sign}_P(\sigma\otimes F_P)={\rm sign}_P(q_P)=0$. If $A_{F_P}$ is not split then ${\rm sign}_P(\sigma)=0$ by \cite[Corollary 11.11]{boi}. Hence ${\rm sign}_P(\sigma)=0$ for all $P\in X_F$.
	\end{proof}
	\medskip

	\begin{lemma}\label{ED-property}
		Let $K/F$ be a totally positive field extension. Let $(A,\sigma)$ be a central simple algebra with an orthogonal involution over $F$ such that $(A,\sigma)_K$ is weakly isotropic over $K$. Then the involution trace form $T_\sigma$ is weakly isotropic over $F$.
	\end{lemma}
	\begin{proof}
		As $(A,\sigma)_K$ is weakly isotropic over $K$ we have $\sum_{i=1}^n\sigma(x_i)x_i=0$ for some nonzero $x_i\in A_K$. Hence $T_\sigma(\sum_{i=1}^n\sigma(x_i)x_i)=0$ which implies $(T_\sigma)_K$ is weakly isotropic. Since $K/F$ is totally positive this implies that $T_\sigma$ is weakly isotropic over $F$.
	\end{proof}

	\begin{remark}
		In general, if $T_\sigma$ is weakly isotropic over $F$ it need not imply that $(A,\sigma)$ is weakly isotropic over $F$. This converse is known if $F$ has the \emph{effective diagonalization property} ED, (defined by Ware  \cite{Ware}). If $F$ has the ED property then by a result of Lewis, Scheiderer and Unger \cite[Theorem 4.3]{LSU}, the converse is true. The following corollary is immediate from their theorem. 
	\end{remark}
	
	\begin{corollary}
		Let $F$ be a formally real field with ED property. Let $K/F$ be a totally positive field extension. Let $(A,\sigma)$ be a central simple algebra with an orthogonal involution over $F$ such that $(A,\sigma)_K$ is weakly isotropic over $K$. Then $(A,\sigma)$ is weakly isotropic over $F$.
	\end{corollary}
	\begin{proof}
		By the above (Lemma \ref{ED-property}), $T_\sigma$ is weakly isotropic over $F$. As $F$ has ED property by \cite[Theorem 4.3]{LSU} $(A,\sigma)$ is weakly isotropic over $F$.
	\end{proof}
	
	\noindent{\bf Proof of Theorem \ref{weak hyp involution}}.
	Let $(A, \sigma)_K$ be weakly hyperbolic. By (Lemma \ref{weak hyperbolic signature}) above, ${\rm sign}_P\ \sigma_K=0$ for all $P \in X_K$. As $K/F$ is totally positive, for every $\widetilde{P}\in X_F$, there exists an ordering ${P} \in X_K$ extending $\widetilde{P}$. As $F_{\widetilde{P}} \subset K_{P}$ we have ${\rm sign}_{\widetilde{P}}\ \sigma= {\rm sign}_{P}\ \sigma_K=0$. Hence ${\rm sign}_{\widetilde{P}}\sigma=0$ for every $\widetilde{P}\in X_F$. This implies that $(A, \sigma)$ is weakly hyperbolic over $F$ by \cite[Theorem 3.2]{ul}.\qed
	\medskip
	
	Recall that $(A,\sigma)$ is said to be \emph{totally decomposable} if \[
	(A, \sigma)\simeq (Q_1,\sigma_1) \otimes_F \cdots \otimes_F (Q_r, \sigma_r)\] for some quaternion algebras $Q_i$ with involution of the first kind $\sigma_i,\ 1\leq i \leq r$. We now prove (Theorem \ref{anologue-for-involution}).
	\medskip
	
	\noindent{\bf Proof of Theorem \ref{anologue-for-involution}}.
	Let $(A,\sigma)$ be weakly isotropic over $K$. Let $T_\sigma$ be the involution trace form of $(A,\sigma)$ over $F$. Then $T_\sigma=T_{\sigma_1}\otimes\cdots\otimes T_{\sigma_r}$. As $T_{\sigma_i}$ is a $2$- fold Pfister form for $1\leq i\leq r$ (see \cite[Proposition 11.6]{boi}), we have $T_{\sigma}$ is similar to a Pfister form.
	
	As $(A,\sigma)_K$ is weakly isotropic over $K$ by (Lemma \ref{ED-property}) $(T_\sigma)_K$ is weakly isotropic.
	We will show that ${\rm sign}_P\ \sigma=0$ for all $P\in X_K$. Let $P\in X_K$. As $T_\sigma$ is weakly isotropic over $K$ it is isotropic over $K_P$. Since $T_\sigma$ is similar to a Pfister form, $T_\sigma$ is hyperbolic over $K_P$. Thus ${\rm sign}_P\ T_\sigma=0$ and hence ${\rm sign}_P\ \sigma=0$. So ${\rm sign}_P\ \sigma=0$ for every $P\in X_K$. Hence $(A,\sigma)_K$ is weakly hyperbolic over $K$, by \cite[Theorem 3.2]{ul}.
	
	
	As $K/F$ is totally positive and $(A,\sigma)_K$ is weakly hyperbolic over $K$ by the above (Theorem \ref{weak hyp involution}) $(A,\sigma)$ is weakly hyperbolic over $F$.
	\qed
	\vspace{.5cm}
	

	\section{Becher's conjecture}\label{Conjecture-cases}
	In this section, we prove some new cases when Becher's conjecture (Conjecture \ref{conjecture}) is true. Let $F$ be a formally real field. Let $A$ be a central simple algebra over $F$ of exponent $2$.
	\medskip
	\noindent{\bf Proof of Theorem \ref{proof-conjecture}}.
	We start with the proof of $(i)$. If $pind(A)=1$ then clearly $pind(A_K)=1$ and we are done in this case. Suppose now that $pind(A)=2$, i.e., $A_{F_{\rm py}}\sim H$ for a quaternion division algebra $H$ over $F_{\rm py}$. Let $q\in W(F_{\rm py})$ be the associated norm form of $H$. Let $P\in X_{F_{\rm py}}$ be an ordering such that $q\neq 0$ over $F_P$. As $P\cap F\in X_F$ and $K/F$ is a totally positive extension, there exists $\widetilde{P}\in X_K$ extending $P\cap F$. Clearly, $F_P\subset K_{\widetilde{P}}$. We will show that $pind(A_K)=2$. If $pind(A_K)=1$ then $q$ is hyperbolic over $K_{\rm py}$. Hence ${\rm sign}_{\widetilde{P}}(q)=0$. On the other hand, $q\neq 0$ over $F_P$, i.e., ${\rm sign}_P(q)\neq 0$. As ${\rm sign}_P(q)={\rm sign}_{\widetilde{P}}(q)$ this is a contradiction. So $pind(A_K)=2$. This proves $(i)$.
	\medskip
	
	We next prove $(ii)$. In view of $(i)$ above, we only consider the case when $pind(A)=4$. So $\Ind(A)=pind(A)=4$. Let $q_A\in W(F)$ be an Albert form associated to $A$. As $pind(A)=4$, $(q_A)_{F_{ \rm py}}$ is anisotropic (see \cite[Chapter 3, Theorem 4.8]{Lam}). In this case, $pind\ A_K \in \{1,2,4\}$. We will show that $pind\ A_K =4$. Suppose that $pind\ A_K < 4$. This implies  $(q_A)_{K_{\rm py}}$ is isotropic. As $K_{\rm py} /F$ is a totally positive extension, we have $q_A$ is weakly isotropic over $F$ (see Lemma \ref{totally positive criterion}). Hence $(q_A)_{F_{ \rm py}}$ is weakly isotropic. As $F_{\rm py}$ is pythagorean, $(q_A)_{F_{ \rm py}}$ is isotropic. This is a contradiction as $(q_A)_{F_{ \rm py}}$ is anisotropic. Hence $pind\ A_K=4$ and so $pind\ A= pind\ A_K$. This completes the proof of $(ii)$.
	\medskip
	
	We now prove $(iii)$. By $(i)$ above we only consider the case when $pind(A)=4$. As in the proof of $(ii)$ above, let $q\in W(F_{\rm py})$ be an Albert form associated to the underlying biquaternion division algebra of $A_{F_{\rm py}}$. If $pind(A_K)<4$, then $q_{K_{\rm py}}$ is isotropic over $K_{\rm py}$. 
	Now since $K/F$ is a totally positive Galois extension, $K_{\rm py}/F_{\rm py}$ is totally positive (see Theorem \ref{arbitrary Galois tp}). Therefore, $q$ is weakly isotropic over $F_{\rm py}$ and hence $q$ is isotropic over $F_{\rm py}$, i.e., $pind(A)<4$, a contradiction. This concludes the proof of $(iii)$.\qed		
	\medskip
	
	Let $A\in{}_2Br(F)$. If $A$ is locally split then by Marshall's result (\cite[Theorem 4, Corollary]{Marshall} and \cite[Theorem 3.4]{Becher1}), $A_{F_{\rm py}}=0$, i.e., $pind(A)=1$. Hence trivially Becher's conjecture (Conjecture \ref{conjecture}) holds for such $A$. We have the following related corollary.
	\medskip
	
	\begin{corollary}\label{nonsplit for every ordering}
		Let $K/F$ be a totally positive field extension and $A$ be a central simple $F$-algebra of exponent $2$. For every ordering $P$ of $F$, if $A\sim (-1,-1)\in Br(F_P)$ then Becher's conjecture (Conjecture \ref{conjecture}) holds.
	\end{corollary}
	\begin{proof}
		Consider $A\otimes_F(-1,-1)\in Br(F)$. By hypothesis $A\otimes_F(-1,-1)=0 \in Br(F_P)$ for every ordering $P\in X_F$. Hence $A\otimes_F(-1,-1)=0\in Br(F_{\rm py})$ by Marshall's result (\cite[Theorem 4, Corollary]{Marshall} and \cite[Theorem 3.4]{Becher1}). Hence $A\sim (-1,-1)\in Br(F_{\rm py})$ i.e., $pind(A)=2$. The result now follows from (Theorem \ref{proof-conjecture}($i$)) above.
	\end{proof}

	\bigskip
	
	\noindent P. Mandal\\
	E-mail: {\ttfamily pbmandal@iitb.ac.in}\\
	Department of Mathematics, Indian Institute of Technology Bombay,\\ Powai, Mumbai-400076, India
	\vspace{.5cm}
	
	\noindent R. Preeti\\
	E-mail: {\ttfamily preeti@math.iitb.ac.in}\\
	Department of Mathematics, Indian Institute of Technology Bombay,\\
	Powai, Mumbai-400076, India
	\vspace{.5cm}
	
	\noindent A. Soman\\
	E-mail: {\ttfamily abhaysoman@uohyd.ac.in}\\
	School of Mathematics and Statistics, University of Hyderabad,\\
	Hyderabad-500046, India

\begin{thebibliography}{[XXX]}
		
		\bibitem[B]{Becher} K. J. Becher, {\it Totally positive extensions and weakly isotropic forms}, Manuscripta Math., $120$($1$): $83-90$, $2006$.
		
		\bibitem[B1]{Becher1} K. J. Becher, {\it Decomposability for division algebras of exponent two and associated forms}, Math. Z., $258$($3$): $691–709$, $2008$.
		
		\bibitem[BLS]{BLS} K. J. Becher, D. B. Leep and C. Schubert, {\it Semiorderings and stability index under field extensions}, Israel J. Math., $199$($2$): $547-566$, $2014$.
		
		\bibitem[D]{D} P. K. Draxl, {\it Skew fields}, volume $81$ of Landon Mathematical Society Lecture Note Series, Cambridge University Press, Cambridge, $1983$.
		
		
		\bibitem[KMRT]{boi}\medspace
		M. A. Knus, A. Merkurjev, M. Rost and J. P. Tignol, {\it The book of involutions}, volume $44$ of American Mathematical Society Colloquium Publications, American Mathematical Society, Providence, RI, $1998$. With a preface in French by J. Tits.
		
		\bibitem[L1]{Lam1} T. Y. Lam, {\it Orderings, valuations and quadratic forms}, volume $52$ of CBMS Regional Conference Series in Mathematics. Published for the Conference Board of the Mathematical Sciences, Washington, DC; by the American Mathematical Society, Providence, RI, $1983$.
		
		\bibitem[L2]{Lam} T. Y. Lam, {\it Introduction to quadratic forms over fields}, volume $67$ of Graduate Studies in Mathematics, American Mathematical Society, Providence, RI, $2005$.
		
		\bibitem[L]{Lang} S. Lang, {\it Algebra}, volume $211$ of Graduate Texts in Mathematics, Springer-Verlag, New York, third edition, $2002$.
		
		\bibitem[LSU]{LSU} D. W. Lewis, C. Scheiderer and T. Unger, {\it A weak Hasse principle for central simple algebras with an involution}, In {\it Proceedings of the Conference on Quadratic Forms and Related Topics (Baton Rouge, LA, 2001)}, number Extra Vol., $241-251$, $2001$.
		
		\bibitem[LU]{ul}\medspace
		D. W. Lewis and T. Unger, {\it A local-global principle for algebras with involution and Hermitian forms}, Math. Z., $244$($3$): $469–477$, $2003$.
		
		\bibitem[M]{Marshall} M. Marshall, {\it Some local-global principles for formally real fields}, Canadian J. Math., $29$($3$): $606–614$, $1977$.
		
		%
		\bibitem[P]{Prestel} A. Prestel, {\it Lectures on formally real fields}, volume $1093$ of Lecture Notes in Mathematics, Springer-Verlag, Berlin, $1984$.
		
		
		\bibitem[S1]{Sal} D. J. Saltman, {\it Lectures on division algebras}, volume $94$ of CBMS Regional Conference Series in Mathematics. Published by American Mathematical Society, Providence, RI; on behalf of Conference Board of the Mathematical Sciences, Washington, DC, $1999$.
		
		\bibitem[S]{Scha} W. Scharlau, {\it Quadratic and Hermitian forms},  volume $270$ of Grundlehren der Mathematischen Wissenschaften [Fundamental Principles of Mathematical Sciences], Springer-Verlag, Berlin, $1985$.
		
		\bibitem[W]{Ware} R. Ware, {\it Hasse principles and the {$u$}-invariant over formally real fields}, Nagoya Math. J., $61$: $117-125$, $1976$.
		
		
	\end{thebibliography}
\end{document}